\documentclass[]{article}
   
\usepackage{graphicx}
\usepackage{xcolor}

\usepackage{amsmath}
\usepackage{amsfonts}
\usepackage{bbm}
\usepackage{amssymb}
\usepackage{amsthm}

\usepackage{hyperref}

\providecommand{\keywords}[1]{{\it Keywords: } #1}
\providecommand{\ams}[1]{{\it AMS Classification: } #1}

\newtheorem{lemma}{Lemma}
\newtheorem{theorem}{Theorem}

\newtheorem{proposition}{Proposition}
\newtheorem{corollary}{Corollary}

\theoremstyle{definition}
\newtheorem{remark}{Remark}
\newtheorem*{conjecture}{Conjecture}

\newcommand{\R}{\mathbb{R}}
\newcommand{\C}{\mathbb{C}}
\renewcommand{\S}{\mathbb{S}}

\renewcommand{\Im}{\mathrm{Im}}
\renewcommand{\Re}{\mathrm{Re}}
\renewcommand{\i}{\mathrm{i}}

\renewcommand{\bar}[1]{\overline{#1}}
\renewcommand{\bf}[1]{\mathbf{#1}}
\newcommand{\bfg}[1]{\boldsymbol{#1}}
\def\diag{\mathrm{diag}}
\newcommand{\norm}[1]{\Vert #1 \Vert}

\begin{document}
  
\title{Minimum trace norm of real symmetric and Hermitian matrices with zero diagonal}

\author{Mostafa  Einollahzadeh}

\date{}
%\date{v4 \\ \large \today }

\maketitle

  \begin{abstract}
  We obtain tight lower bounds for the trace norm \mbox{$\norm{\cdot}_1$} of some matrices with diagonal zero in terms of the entry-wise $L^1$-norm (denoted by $\norm{\cdot}_{(1)}$). It is shown that in the space of nonzero real symmetric matrices $A$ of order $n$ with diagonal zero, the minimum value  of quantity $\frac{\norm{A}_1}{\norm{A}_{(1)}}$ is equal to $\frac{2}{n}$. The answer to a similar problem in the space of Hermitian matrices is also obtained to be equal to $\tan(\frac{\pi}{2n})$. As an equivalent form of these results, it is shown that for every real symmetric matrix $A$ of order $n$ with off-diagonal entries of absolute values at most $1$, there always exists a diagonal matrix $D$ that is  at distance  $\leq\frac{n}{2}$  of $A$ in the spectral norm. The similar problem for Hermitian matrices, has answer $\cot(\frac{\pi}{2n})$.  
  \end{abstract}

  \ams{15A42, 05C50}


\section{Introduction}

The trace norm of a matrix (denoted by $\norm{\cdot}_1$) is defined as the sum of its singular values and the spectral norm (denoted by $\norm{\cdot}_\infty$) is defined as the maximum singular value. In particular, in the case of Hermitian matrices or real symmetric matrices, which we consider in this study, the singular values are equal to the absolute values of the eigenvalues, which can be used to define the above norms. The trace norm is also known by ``energy" in some references, mostly those related to algebraic graph theory or its applications in chemistry (cf. \cite{gutman} for a survey).

The main idea of this paper comes from a conjecture by Haemers in  \cite{haemers2012seidel} (proved in\cite{akbari2020proof}), on ``the minimum energy of Seidel matrix of simple graphs", which can be restated in our terminology as follows:

\begin{conjecture}
Let $A$ be an arbitrary symmetric matrix of order $n$, with diagonal entries equal to zero and off-diagonal entries in $\{\pm 1\}$. Then the minimum value of $\norm{A}_1$ is equal to $\norm{J_n-I_n}_1=2n-2$ for the all-ones matrix $J_n$ and the identity matrix $I_n$, both of order $n$.
\end{conjecture}

Numerical computations for small values of $n$, suggest that the matrix $J_n-I_n$ has the minimum trace norm in the set of all diagonal zero matrices of order $n$ for which the sum of the absolute values of the entries is equal to $n(n-1)$. Therefore, the above conjecture on the discrete set of matrices with $\{\pm 1\}$ off-diagonal entries can be generalized to provide a lower bound for the trace norm of every diagonal zero matrix with a prescribed value for the $L^1$-norm of its entries. This is one of the main results of this paper:

\begin{theorem}\label{real}
Let $A=[a_{ij}]$ be a real symmetric  matrix of order $n$ with zero entries on the main diagonal. Then 
\[ \norm{A}_1\geq \frac{2}{n}\sum_{i,j} |a_{ij}|,\]
and the bound is sharp for all $n$.
\end{theorem}

A similar problem can be imposed on other spaces of matrices. We also obtained the best result for the space of Hermitian matrices:

\begin{theorem}\label{complex}
Let $A=[a_{ij}]$ be a Hermitian matrix of order $n$ with zero entries on the main diagonal. Then 
\[ \norm{A}_1\geq \tan(\frac{\pi}{2n})\sum_{i,j} |a_{ij}|,\]
and the bound is sharp for all $n$.
\end{theorem}

By considering the standard Frobenius inner product on the space of real symmetric or Hermitian matrices, we obtained equivalent ``dual" results of the above theorems. The dual results are related to the minimum distance in the spectral norm of a matrix to the space of diagonal matrices.

We denote the space of all real diagonal matrices of order $n$ by $\mathcal{D}_n$. Then, the first ``dual" theorem in the case of real symmetric matrices can be expressed as follows:
\begin{theorem}\label{realdual}
Let $A=[a_{ij}]$ be a real symmetric matrix of order $n>1$. Then
\[\min_{D\in \mathcal{D}_n} \norm{A-D}_{\infty} \;\leq\; \frac{n}{2}\max_{i\neq j} |a_{ij}|,\]
and the bound is sharp for all $n>1$.
\end{theorem}
The equality case in the above theorem is obtained by the matrix $A=J_n$, because $\norm{J_n-\frac{n}{2}I_n}_\infty=\frac{n}{2}$ and $\frac{n}{2}I_n$ is the nearest diagonal matrix to $J_n$ in the spectral norm. It is interesting to note that the worst case in this respect is the simple matrix $J_n$.

It is worth noting that there is no effective method for finding a diagonal matrix at spectral distance $\frac{n}{2}$ from a given matrix with entries in $[-1,1]$, as we know its existence from the above theorem. This is an interesting problem on its own.

The final theorem in this study is the Hermitian version of Theorem \ref{realdual}. 
\begin{theorem}\label{complexdual}
Let $A=[a_{ij}]$ be a Hermitian matrix of order $n>1$. Then
\[\min_{D\in \mathcal{D}_n} \norm{A-D}_{\infty} \;\leq\; \cot(\frac{\pi}{2n})\max_{i\neq j} |a_{ij}|,\]
and the bound is sharp for all $n>1$.
\end{theorem}

The remainder of this paper is organized as follows: Section 1 is devoted to the basic definitions and known results that are used throughout. In Sections 2 and 3, we present the proofs and necessary lemmas for Theorems \ref{real} and \ref{complex}, respectively. Section 4 is devoted to the dual versions, Theorems \ref{realdual} and \ref{complexdual}.  
\section{Notation and preliminaries}

In the first part of this section, we review some of the basic concepts and results of the theory of matrix analysis for the reader's convenience and fixing notations. Detailed discussion and proofs can be found in any standard textbook on this subject, such as \cite{bhatia}, \cite{horn} and \cite{zhan}.

The  conjugate transpose of a general complex matrix $A$ is denoted by $A^*$. For two complex vectors $\bf{u}$ and $\bf{v}$ of the same dimension, the usual Euclidean inner product is denoted by $\langle u,v\rangle := \bf{u}^*\bf{v}$ and the corresponding norm by $|\bf{u}|:=\sqrt{\langle \bf{u},\bf{u}\rangle}$.

The spaces of real and complex $m\times n$ matrices are denoted by $M_{m\times n}(\R)$ and $M_{m\times n}(\C)$, respectively. The space of real symmetric matrices of order $n$ is denoted by $\mathcal{S}_n$ and the space of Hermitian matrices of order $n$ is denoted by $\mathcal{H}_n$.  The space of real diagonal matrices of order $n$ is also denoted by $\mathcal{D}_n$.

Let $A\in \mathcal{H}_n$ has eigenvalues $\lambda_1,\dots,\lambda_n$. The \textit{trace norm} (energy, nuclear norm, ...) of $A$ is defined by $\norm{A}_1 = \sum_{i=1}^n |\lambda_i|$ and the \textit{spectral norm} of $A$ is defined by $\norm{A}_\infty=\max_i |\lambda_i|$.
\footnote{These are two special cases of the Schatten $p$-norms defined  by the $L^p$-norm of the singular values of a general matrix, which is the reason for choosing the present notation.} The two other \textit{entry-wise  $L^1$-norm} and \textit{max norm} of $A$ are also defined by  the formulas $\norm{A}_{(1)}=\sum_{i,j} |A_{ij}|, \norm{A}_{(\infty)}=\max_{i,j} |A_{ij}|$. We also denote vector $(A_{11},A_{22},\dots,A_{nn})$ of the diagonal entries of $A$ by $\mathrm{diag}(A)$.

Now let
\[ A= \lambda_1 \bf{x}_1\bf{x}^*_1+\dots+\lambda_n \bf{x}_1\bf{x}^*_n,\]
be the spectral decomposition of $A$ for an orthonormal basis   
 $\bf{x}_1,\dots,\bf{x}_n$ of the eigenvectors and corresponding eigenvalues $\lambda_1,\dots,\lambda_n$. Suppose for some $k$,
 \[\lambda_1,\dots,\lambda_k\geq 0, \quad\lambda_{k+1},\dots,\lambda_n <0 .\]
 
   The \textit{positive }and \ textit{negative}parts of $A$ are defined by
\[A^+:= \sum_{i\leq k} \lambda_i \bf{x}_i\bf{x}^*_i,\quad A^-:=-\left(\sum_{i>k} \lambda_i \bf{x}_i\bf{x}^*_i\right).\]

It can be easily seen that $A^+,A^-$ are well-defined positive semi-definite matrices, and $A=A^+-A^-$.

We also use the Frobenius inner product on spaces $M_{m\times n}(\C)$. For two matrices $A,B\in M_{m\times n}(\C)$, their (Frobenius) \textit{inner product} is defined by:
\[\langle A,B\rangle := \mathrm{tr}(A^*B)=\sum_{i,j} \bar{A_{ij}}B_{ij}.\]
It can be verified that this yields a real inner product on $\mathcal{H}_n$.

We require two basic properties of the Frobenius inner product:
\begin{itemize}
\item For every unitary matrix $U$ and square matrices $A$ and $B$, all of the same order,
\[\langle A,B\rangle = \langle U^*AU,U^*BU\rangle.\]
\item The inner product of two positive semi-definite matrices of the same order, is nonnegative.
\end{itemize}

For two matrices $A,B\in M_{m\times n}(\C)$, the \textit{Hadamard product} $A\circ B$ is a matrix of the same dimension $m\times n$, with elements given by
\[(A\circ B)_{ij} = A_{ij}B_{ij}.\] 
The following lemma is known as the ``Schur product theorem" (\cite{horn}, Theorem 7.5.3): 

\begin{lemma}\label{plem}
The Hadamard product of two positive semi-definite matrices is positive semi-definite.
\end{lemma}

In the second part of this section, we review the ``duality transformation" on convex sets.  For more information on this subject, the reader can consult 
\cite{matousek} (for the duality transformation in general ) and \cite{zhan} (for dual norms).

Let $(V,\langle\cdot,\cdot\rangle)$ be a real inner product space of finite dimension. For every subset $X$ of $V$, we define the \textit{dual} (or polar) of $X$ by,
\[X^* = \{y\in V: \langle y,x\rangle \leq 1, \forall x\in X\}.\]

We denote the convex hull and closure of set $X$ by $\mathrm{conv}(X)$ and $\mathrm{cl}(X)$, respectively. In the next lemma, we list some of the basic properties of the duality transform required in this study.

\begin{lemma}\label{plem2} Let $V$ be a finite dimensional real inner product space, and $X,Y$ be subsets of $V$.
\begin{enumerate}
\item $X^*$ is a closed convex set and contains $0$.
\item If $X\subset Y$, then $Y^*\subset X^*$.
\item $X^*=\mathrm{conv}(X)^*=\mathrm{cl}(X)^*$.
\item $(X\cup Y)^*$= $X^*\cap Y^*$.
\item Let $W$ be a linear subspace of $V$. Subsequently, $W^*=W^{\perp}$ for the orthogonal complement $W^\perp$ of $W$. 
\item If $X$ is closed convex and contains $0$, then $(X^*)^*=X$.
\item If $X$ and $Y$ are closed convex and contain $0$, and $Z=\mathrm{cl}(\mathrm{conv}(X^*\cup Y^*))$, then
\[(X\cap Y)^* = Z.\]
\end{enumerate}
\end{lemma}
\begin{proof}
The statements 1-5, are immediate consequences of the definition, and 6 is standard (cf. \cite{matousek},  Section 5.1).

For 7, first by assumption we have $X=(X^*)^*$ and $Y=(Y^*)^*$ and by double use of 3, we have,
\[Z^*= (X^*\cup Y^*)^*= (X^*)^*\cap (Y^*)^*
= X\cap Y.\]
Now $Z$ is closed convex and contains $0$ and so is equal to its second dual, which gives,
\[(X\cap Y)^*= (Z^*)^*=Z.\]
\end{proof}

For every norm $\norm{\cdot}$ on $V$, its \textit{dual norm} $\norm{\cdot}'$ is defined by
\[\norm{v}' = \max \{\langle v,w\rangle, \norm{w}\leq 1\}.\]  
Equivalently, the unit balls of $\norm{\cdot}$ and $\norm{\cdot}'$ are dual.

In the spaces $\mathcal{H}_n$ and $\mathcal{S}_n$ with the Frobenius inner product (which is the only inner product we consider in this study), the dual norms of $\norm{\cdot}_1$ and $\norm{\cdot}_{(1)}$ are $\norm{\cdot}_\infty$ and $\norm{\cdot}_{(\infty)}$ respectively (cf. \cite{zhan}, Section 4.3).

\section{Real symmetric matrices}

In this section, we prove Theorem \ref{real} by using two lemmas.

\begin{lemma}\label{lem1}
Let $d>0$ and  $a,b\in \R$ with $|a|,|b|\leq d$. Then
\[|a-b| \leq \left( d-\frac{ab}{d}\right) .\] 
\end{lemma}

\begin{proof}
We have
\[ \left(d-\frac{ab}{d}\right)-(a-b)=\frac{1}{d}(d+b)(d-a)\geq 0, \]
Thus, $(d-\frac{ab}{d})\geq(a-b)$ and similarly $(d-\frac{ab}{d})\geq(b-a)$.
\end{proof}

\begin{proposition}\label{prop1}
Let $A,B$ be two real positive semi-definite matrices of order $n$, with 
\[\mathrm{diag}(A)=\mathrm{diag}(B)=(d_1,\dots,d_n).\]
Then
\[\sum_{i,j} |A_{ij}-B_{ij}|\leq \left( \sum_{i=1}^n \sqrt{d_i} \right)^2.\]
\end{proposition}

\begin{proof}
$A$ and $B$ are positive semi-definite, so all of their $2\times 2$ principal minors are nonnegative, i.e. for all $i,j$ we have
\[d_id_j - (A_{ij})^2 = A_{ii}A_{jj}-(A_{ij})^2 \geq 0 \Longrightarrow\ |A_{ij}|\leq \sqrt{d_id_j}.\]
Similarly $|B_{ij}|\leq \sqrt{d_id_j}$.

Let $C=A\circ B$ and define the vector $\bf{e}=(e_1,\dots,e_n)$ by the equations

\[e_i = \left\{ \begin{array}{ll} \frac{1}{\sqrt{d_i}} & d_i>0 \\ 0 & d_i=0 \end{array}\right. .\]

For every $i\neq j$, if $d_id_j=0$,
$ A_{ij}=B_{ij}=C_{ij}=0,$
and if $d_id_j>0$, $|A_{ij}|,|B_{ij}|\leq \sqrt{d_id_j}$. So by Lemma \ref{lem1}, for all $i,j$,
\begin{equation}
|A_{ij}-B_{ij}|\leq \left(\sqrt{d_id_j} - e_ie_jC_{ij} \right).
\end{equation}.

Summing up the above inequalities, we obtain
\begin{align}
\sum_{i,j}|A_{ij}-B_{ij}| & 
\leq \left( \sum_i\sqrt{d_i}\right)^2- \sum_{i,j} e_ie_jC_{ij} \nonumber\\
&= \left( \sum_i\sqrt{d_i}\right)^2-\bf{e}^TC\bf{e}.
\end{align}

By Lemma \ref{plem}, $C$ is also positive semi-definite, hence \begin{equation}
\bf{e}^TC\bf{e}\geq 0.
\end{equation}
This completes the proof.  
\end{proof}

\begin{proof}[Proof of Theorem \ref{real}]
We have $A=A^+-A^-$ and $\mathrm{diag}(A)=0$, so for a vector $(d_1,\dots,d_n)\in\R^n$,
\[\mathrm{diag}(A^+)=\mathrm{diag}(A^-)=(d_1,\dots,d_n).\]

Let $\lambda_1,\dots,\lambda_k\geq 0$ be all non-negative eigenvalues and $\lambda_{k+1},\dots,\lambda_n$ all negative eigenvalues of $A$. Therefore
\begin{equation}
\norm{A}_1=\sum_{i\leq k} \lambda_i -\sum_{i>k} \lambda_i = \mathrm{tr}(A^+)+\mathrm{tr}(A^-) 
= 2 \sum_i d_i.
\end{equation}

Now by Proposition \ref{prop1},
\begin{equation}
\sum_{i,j} |A_{ij}| = \sum_{i,j} |A^+_{ij}-A^-_{ij}| 
\leq \left( \sum_i \sqrt{d_i}\right)^2 \leq n \sum_i d_i = \frac{n}{2}\norm{A}_1.
\end{equation}

For an equality case, consider the matrix $A=J_n-I_n$ for the all-one matrix $J_n$ and the identity matrix $I_n$, both of order $n>1$.  The eigenvalues of $A$ are $n-1$ and $-1$ with multiplicity $n-1$. Hence $\frac{\norm{A}_1}{\norm{A}_{(1)}}=\frac{2(n-1)}{n(n-1)}=\frac{2}{n}$.
\end{proof}

\begin{remark}\label{remreal}
For more general example of matrices $A$, which give equality in Theorem \ref{real}, consider the matrix, 
\[ A=\mathbbm{1}\mathbbm{1}^*-B,\]
for $\mathbbm{1}=(1,1,\dots,1)\in\R^n$ and arbitrary positive semi-definite matrix $B\in\mathcal{S}_n$, with condition
\[B\mathbbm{1}=0, \quad \diag(B)=(1,1,\dots,1).\]
 
Then $\langle \mathbbm{1}\mathbbm{1}^*,B\rangle=0$, and $A^+=\mathbbm{1}\mathbbm{1}^*, A^-=B$. So we have, $\diag(A)=0, \norm{A}_1=2\mathrm{tr}(A^+)=2n$ and because $|B_{ij}|\leq 1$ for all $i,j$ (by $B\geq 0$),   
\[\norm{A}_{(1)} = \sum_{i,j} |1-B_{ij}| = n^2-\sum_{i,j} B_{ij}=n^2.\]
Therefore $\frac{\norm{A}_1}{\norm{A}_{(1)}}=\frac{2}{n}$.

Also it can be easily seen that, for every $d\neq 0$ and every diagonal matrix $D$ with diagonal entries in $\{\pm 1\}$, the matrix $d DAD$ is also an example of equality in Theorem \ref{real}. These matrices constitute the set of all zero diagonal matrices in $\mathcal{S}_n$ which have only one positive or negative eigenvalue with the corresponding eigenvector in the set $\{\pm 1\}^n$. It seems that, these are all of the nonzero equality cases, although we don't have a proof.  
\end{remark}

\section{Hermitian matrices}

In the previous section, we observed that the minimum of $\frac{\norm{A}_{1}}{\norm{A}_{(1)}}$ over all nonzero real symmetric  matrices $A$ of order $n$ with zero entries on the diagonal, is equal to $\frac{2}{n}$. In this section, we present a proof of Theorem \ref{complex}, which provides an answer to the similar problem for Hermitian matrices. It turns out that for the constant $\gamma_n$ defined below, the answer is equal to $\frac{1}{\gamma_n}\cdot\frac{2}{n}$.

We define:
\begin{equation}
\gamma_n := \frac{2}{n} \sum_{k=0}^{n-1} \sin(\frac{k\pi}{n}).
\end{equation}

For $n\geq 2$, there is a simpler formula for $\gamma_n$:
\begin{align}
\gamma_n  &= \frac{2}{n}\, \Im \left( \sum_{k=0}^{n-1} \exp(\i\frac{k\pi}{n})\right) 
= \frac{2}{n}\, \Im \left( \frac{1-\exp(\i\pi)}{1-\exp(\frac{\i\pi}{n})}\right) \nonumber\\
&= \frac{4}{n} \cdot\frac{\sin(\frac{\pi}{n})}{2-2\cos(\frac{\pi}{n})} 
= \frac{2}{n} \cdot\frac{2\sin(\frac{\pi}{2n})\cos(\frac{\pi}{2n})}{2\sin^2(\frac{\pi}{2n})} \nonumber\\
&= \frac{2}{n\tan(\frac{\pi}{2n})}.
\end{align}

and because $\frac{\tan(x)}{x}$ is increasing in the interval $[0,\frac{\pi}{2})$ and $\lim_{x\to 0}\frac{\tan(x)}{x}=1$,
\begin{equation}\label{gamma}
\gamma_1=0\leq \gamma_2=1 \leq \gamma_3 \leq \dots, \quad \lim_{n\to \infty} \gamma_n = \frac{4}{\pi}.
\end{equation}

\begin{lemma}\label{clem1}
Let $d>0$ and  $a,b\in \C$ with $|a|,|b|\leq d$. Then
\[|a-b| \leq \left| d-\frac{\bar{a}b}{d}\right| .\] 
\end{lemma}

\begin{proof}
We have

\begin{align*}
\left| d-\frac{\bar{a}b}{d}\right|^2-|a-b|^2 &=
\left(d^2+\frac{|a|^2|b|^2}{d^2}-\bar{a}b-a\bar{b}\right)   - \left(|a|^2+|b|^2-\bar{a}b-a\bar{b}\right) \\
&= \frac{1}{d^2}(d^2-|a|^2)(d^2-|b|^2)\geq 0.
\end{align*}
Hence, $|a-b|^2 \leq \left| d-\frac{\bar{a}b}{d}\right|^2$ as required.
\end{proof}

%%%%%%%%%%%%clem2%%%%%%%%%%%%%%%%%%%%

Hereafter, we use notation $\S^1=\{z\in \C: |z|=1\}$.
\begin{lemma}\label{clem2}
Let $d_1,\dots,d_n \geq 0$ and $\omega_1,\dots,\omega_n\in \S^1$. Then
\[\sum_{i,j} d_id_j |\omega_i-\omega_j| \,\leq\, \gamma_n \left( \sum_i d_i\right)^2,\]
and equality holds if $\{\omega_1,\dots,\omega_n\}$ is the set of $n$-th roots of unity and {$d_1=d_2=\dots=d_n$}.
\end{lemma}

Because the proof is rather long, we postpone it until Section \ref{clem2sec}.

\begin{corollary}\label{cor1}
For every vector $\bf{x}=(x_1,\dots,x_n)\in \C^n$,
\[\sum_{i,j} \left| |x_ix_j| - x_i\bar{x_j}\right| \leq \gamma_n \left(\sum_i|x_i| \right)^2.\]
\end{corollary}
\begin{proof}
Apply the previous lemma to $d_i=|x_i|$ and $x_i=d_i\omega_i$ for $\omega_i\in \S^1$. Then we have
\begin{align*}
\sum_{i,j} \left| |x_ix_j| - x_i\bar{x_j}\right|  
&= \sum_{i,j} \left| d_id_j - d_id_j\omega_i\bar{\omega_j}\right| \\
&= \sum_{i,j} d_id_j\left| \omega_i - \omega_j\right| \\
&\leq \gamma_n \left( \sum_i d_i\right)^2 = \gamma_n \left(\sum_i|x_i| \right)^2.
\end{align*}

\end{proof}

%%%%%%%%%%%%cprop%%%%%%%%%%%%%%%%

\begin{proposition}\label{cprop}
Let $A,B$ be two Hermitian semi-definite matrices of order $n$, with 
\[\mathrm{diag}(A)=\mathrm{diag}(B)=(d_1,\dots,d_n).\]
Then
\[\sum_{i,j} |A_{ij}-B_{ij}|\leq \gamma_n\left( \sum_{i=1}^n \sqrt{d_i} \right)^2.\]
\end{proposition}

\begin{proof}
Define the vector $\bf{e}\in \C^n$ and the Hermitian matrix $C$ by the equations,
\[e_i = \left\{ \begin{array}{ll} \frac{1}{\sqrt{d_i}} & d_i>0 \\ 0 & d_i=0 \end{array}\right. ,
\quad
C_{ij}=e_ie_j\bar{A_{ij}}B_{ij}.\]

Then $\diag(C)=(d_1,\dots,d_n)$ and by semi-positivity of $A$ and $B$, for all $i,j$,
\[|A_{ij}|,|B_{ij}|\leq \sqrt{d_id_j},\]
and if $d_id_j=0$, $A_{ij}=B_{ij}=C_{ij}=0$.
Hence by Lemma \ref{clem1},
\begin{equation}\label{ineq1}
\sum_{i,j} |A_{ij}-B_{ij}| \leq \sum_{i,j} |\sqrt{d_id_j}-C_{ij}|.
\end{equation}

Let $X$ be a diagonal matrix with $e_1,\dots,e_n$ along its diagonal. Then
\[C= X^*\, (\bar{A}\circ B)\, X,\] and by Lemma \ref{plem}, 
 $ \bar{A}\circ B$ 
and also $C$ are positive semi-definite. Hence for some vectors $\bf{x}_1,\dots,\bf{x}_k\in \C^n$, we can write,
\[C = \bf{x}_1\bf{x}_1^*+\cdots+\bf{x}_k\bf{x}_k^*.\]

For every $i\leq k$, define $\bf{y}_i\in\R_{\geq 0}^n$ by,
\[(\bf{y}_i)_j=|(\bf{x}_i)_j|, \quad 1\leq j\leq n.\]
Now if we define 
\[D= \bf{y}_1\bf{y}_1^*+\cdots+\bf{y}_k\bf{y}_k^*,\]
we have $\diag(D)=\diag(C)=(d_1,\dots,d_n)$ and $D$ is also positive semi-definite.
So for all $i,j$, $|D_{ij}|\leq \sqrt{d_id_j}$, and by triangle inequality
\begin{equation}\label{ineq2}
\sum_{i,j} |\sqrt{d_id_j}-C_{ij}| 
\leq \sum_{i,j} \left( \sqrt{d_id_j}-D_{ij}+|D_{ij}-C_{ij}|\right).
\end{equation}

By the definition of $C$ and $D$ and applying Corollary \ref{cor1} to vectors
$\bf{x}_1,\dots,\bf{x}_k$, we also have
\begin{align}\label{ineq3}
\sum_{i,j} |D_{ij}-C_{ij}| &\leq 
\sum_{s=1}^k \sum_{i,j} |(\bf{y}_s)_i(\bf{y}_s)_j-(\bf{x}_s)_i\bar{(\bf{x}_s)_j}| \nonumber\\
&=  
\sum_{s=1}^k \sum_{i,j} |
|(\bf{x}_s)_i(\bf{x}_s)_j|-(\bf{x}_s)_i\bar{(\bf{x}_s)_j}| \nonumber\\
&\leq \gamma_n\sum_{s=1}^k \left(\sum_{i} (\bf{y}_s)_i\right)^2 \nonumber\\
&=  
\gamma_n\sum_{i,j} D_{ij}.
\end{align}
 
Finally combining inequalities \eqref{ineq1}, \eqref{ineq2} and \eqref{ineq3}, and the fact that $\gamma_n\geq 1$ for $n\geq 2$, we have 
\begin{align*}
\sum_{i,j} |A_{ij}-B_{ij}| &\leq 
 \sum_{i,j} (\sqrt{d_id_j}-D_{ij}) +
\gamma_n\sum_{i,j} D_{ij} \\
&\leq 
 \gamma_n\sum_{i,j} (\sqrt{d_id_j}-D_{ij}) +
\gamma_n\sum_{i,j} D_{ij} \\
&=\gamma_n \left( \sum_i \sqrt{d_i}\right)^2.
\end{align*}

\end{proof}

\begin{proof}[Proof of Theorem \ref{complex}]
This proof is similar to that of Theorem \ref{real}. If $\diag(A^+)=(d_1,\dots,d_n)$, we have
\begin{equation}
\norm{A}_1 = \mathrm{tr}(A^+)+\mathrm{tr}(A^-) = 2 \sum_i d_i,
\end{equation}
and by applying Proposition \ref{cprop} to $A^+$ and $A^-$,
\begin{align}\label{comst}
\sum_{i,j} |A_{ij}| &= \sum_{i,j} |A^+_{ij}-A^-_{ij}| \,\leq
 \gamma_n \left( \sum_{i=1}^n \sqrt{d_i} \right)^2 \nonumber\\
 &\leq n\gamma_n\sum_i d_i 
 \,=\frac{n\gamma_n}{2}\norm{A}_1 \,= \cot(\frac{\pi}{2n})\norm{A}_1.
 \end{align}

For an equality case, let $A$ be given by $\mathbbm{1}\mathbbm{1}^*-\bfg{\alpha}\bfg{\alpha}^*$ for the same notation defined in \eqref{1alpha}. It can be easily checked that $A$ is a Hermitian matrix with diagonal zero, and because 
\[\ \langle\mathbbm{1},\bfg{\alpha}\rangle=0,\quad  |\mathbbm{1}|^2=|\bfg{\alpha}|^2=n, \] 
it has two nonzero eigenvalues $n,-n$, and $\norm{A}_1=2n$. On the other hand, the equality statement of Lemma \ref{clem2}, gives,
\begin{align*}
\sum_{i,j}|A_{ij}| = \sum_{i,j} |1-\alpha_i\bar{\alpha_j}| 
= \sum_{i,j} |\alpha_i-{\alpha_j}| 
= n^2 \gamma_n.
\end{align*} 
Hence, $\frac{\norm{A}_1}{\norm{A}_{(1)}} = \frac{2n}{n^2\gamma_n}=\tan(\frac{\pi}{2n})$, which yields the desired equality for \eqref{comst}.
 
\end{proof} 
 \begin{remark}
From the relations in \eqref{gamma}, the sequence $\gamma_n = \frac{2}{n\tan(\frac{\pi}{2n})}$ converges  from below to $\frac{4}{\pi}$. So by Theorem \ref{complex}, for every nonzero Hermitian matrix $A$ of order $n$ with zero diagonal, we have $\norm{A}_{1} \geq \frac{2}{n\gamma_n} \norm{A}_{(1)}$ and so
\begin{equation}
\norm{A}_{1} > \;\frac{\pi}{2}\cdot \frac{1}{n} \norm{A}_{(1)},
\end{equation}
and $\frac{\pi}{2}$ is the best constant (independent of $n$) for which the above inequality holds for all $n$.
 \end{remark}

\section{The dual version}

In this section, we prove Theorems \ref{realdual} and \ref{complexdual}. The main idea is to view the inequality statements in Theorems \ref{real} and \ref{complex} as statements regarding the inclusion of certain convex sets. Then, the duality transformation is applied to these sets to obtain reverse inclusion relations between the resulting dual sets, and these inclusions are finally interpreted as a matrix norm inequality.

For every subscript $*$, we denote the closed unit ball of the norm $\norm{\cdot}_*$ by $\mathcal{B}_*$. Here, $\norm{\cdot}_*$ can be any of the norms $\norm{\cdot}_1,\norm{\cdot}_\infty,\norm{\cdot}_{(1)},\norm{\cdot}_{(\infty)}$ on the underlying spaces $\mathcal{S}_n$ or $\mathcal{H}_n$, and the underlying space is known from the context. All these balls are bounded closed sets and are consequently compact.

\begin{lemma}
Let $V$ be a finite dimensional real inner product space, and $X$ be a compact convex subset of $V$ that contains $0$. Then for every linear subspace $L$ of $V$,
\[\mathrm{cl}(\mathrm{conv}(X\cup L))=X+L.\]
\end{lemma}

\begin{proof}
Because of the compactness of $X$, it can be easily seen that $X+L$ is closed. It is also convex and, therefore, contains $\mathrm{cl}(\mathrm{conv}(X\cup L))$. Next, we demonstrate that every point in $X+L$ is a limit point of $\mathrm{conv}(X\cup L)$, which completes the proof.

Let $x+v$ (for $x\in X$ and $v\in L$) be an arbitrary point in $X+L$. Then we have
\[x+v = \lim_{t\to 1^-} tx+(1-t)\left(\frac{v}{1-t}\right).\]
Clearly, $tx+(1-t)\left(\frac{v}{1-t}\right)\in \mathrm{conv}(X\cup L)$ for $0\leq t<1$; thus, $x+v$ is a limit point of $\mathrm{conv}(X\cup L)$.
\end{proof}

\begin{proof}[Proof of Theorem \ref{realdual}]
First, note that the orthogonal complement of $\mathcal{D}_n$ is equal to the set of matrices with diagonal zero. Now let $A$ be an arbitrary real symmetric matrix in $\mathcal{B}_1\cap \mathcal{D}_n^\perp$. Then, $\norm{A}_1\leq 1$ and $A$ has zero entries on the diagonal. So by Theorem \ref{real},
$\norm{A}_{(1)} \leq \frac{n}{2}$, or equivalently, $A\in \frac{n}{2}\mathcal{B}_{(1)}$.
Therefore the inequality statement of Theorem \ref{real} implies the inclusion,
\begin{equation}\label{dualreal1}
\mathcal{B}_1 \cap \mathcal{D}_n^\perp \subset \frac{n}{2}\mathcal{B}_{(1)}.
\end{equation}

By taking dual of both sides, we obtain
\begin{equation}\label{dualeq}
\frac{2}{n}\mathcal{B}_{(1)}^* \subset \left( \mathcal{B}_1\cap \mathcal{D}_n^\perp\right)^*.
\end{equation}

Now note that the norms $\norm{\cdot}_1$ and $\norm{\cdot}_\infty$, are dual to each other, and so by definition,
\[\mathcal{B}_1^* = \mathcal{B}_{\infty}. \]
Similarly for the dual norms $\norm{\cdot}_{(1)}$ and $\norm{\cdot}_{(\infty)}$, we have
\[\mathcal{B}_{(1)}^* = \mathcal{B}_{(\infty)}. \]

The ball $\mathcal{B}_1$ is closed convex and contains zero, so by Lemma \ref{plem2}, 
\[\left( \mathcal{B}_1\cap \mathcal{D}_n^\perp\right)^*
=\mathrm{cl}(\mathrm{conv}(\mathcal{B}_1^* \cup (\mathcal{D}_n^\perp)^*))
=\mathrm{cl}(\mathrm{conv}(\mathcal{B}_\infty \cup \mathcal{D}_n)).\]
Using the previous lemma for $X=\mathcal{B}_\infty$ and $L=\mathcal{D}_n$, we have
\[\mathrm{cl}(\mathrm{conv}(\mathcal{B}_\infty \cup \mathcal{D}_n))=\mathcal{B}_\infty+\mathcal{D}_n.\]

Substituting the above relations in the inclusion \eqref{dualeq},and multiplication of the both sides by $\frac{n}{2}$ gives,
\begin{equation}\label{dualreal2}
\mathcal{B}_{(\infty)} \;\subset\; \frac{n}{2} \mathcal{B}_\infty+\mathcal{D}_n.
\end{equation}

This inclusion can also be easily converted into an equivalent matrix-norm inequality. Let $A=[a_{ij}]\in \mathcal{S}_n-\mathcal{D}_n$, and $M= \max_{i\neq j} |a_{ij}|$. Define the matrix $B$ as below,
\[B_{ij}=\left\{\begin{array}{ll}
0 & i=j \\ \frac{a_{ij}}{M} & i\neq j.
\end{array} \right.\]
Then for every $i\neq j$, $|B_{ij}|\leq 1$, and so $B\in\mathcal{B}_{(\infty)}$. In addition, $A=MB+D_1$ for a diagonal matrix $D_1$. Now, from \eqref{dualreal2}, $B=C+D_2$ for some $C\in \mathcal{S}_n$ with $\norm{C}_\infty\leq \frac{n}{2}$ and a diagonal matrix $D_2$. Therefore,

\[ \norm{A-(D_1+MD_2)}_\infty 
= \norm{MC}_\infty\leq \frac{n}{2}\max_{i\neq j} |a_{ij}|,\]
and hence, 
\begin{equation}\label{dualreal3}
\min_{D\in\mathcal{D}_n} \norm{A-D}_\infty \leq\frac{n}{2}\max_{i\neq j} |a_{ij}|.
\end{equation}

In the case of diagonal matrices $A$, the inequality above is trivial; therefore, it is valid for every $A\in \mathcal{S}_n$.

For the sharpness of the inequality, note that all of the above arguments  are reversible for any constant in place of $\frac{n}{2}$. Therefore, if inequality \eqref{dualreal3} remains valid for some other constant   smaller than $\frac{n}{2}$, this is also the case for Theorem \ref{real}, which we know is impossible. 
\end{proof}

\begin{remark}\label{remrealdual} By the equality 
\[\left( \mathcal{B}_1\cap \mathcal{D}_n^\perp\right)^*=
\mathcal{B}_\infty+\mathcal{D}_n,\]
 in above proof, for every matrix $A=[a_{ij}]\in\mathcal{S}_n$,
\begin{align*}
\min_{D\in \mathcal{D}_n} \norm{A-D}_\infty 
&= \min\{r\geq 0: A\in r\mathcal{B}_\infty+\mathcal{D}_n\} \\
&= \min\{r\geq 0: A\in r\left( \mathcal{B}_1\cap \mathcal{D}_n^\perp\right)^*\} \\
&= \max_{B\in \mathcal{B}_1\cap \mathcal{D}_n^\perp} \langle A,B\rangle.
\end{align*}
Therefore if $A'$ be the matrix obtained from $A$ by just changing the diagonal entries to zero, we have
\begin{align}
\min_{D\in \mathcal{D}_n} \norm{A-D}_\infty 
&= \max_{B\in \mathcal{B}_1\cap \mathcal{D}_n^\perp} \langle A,B\rangle \nonumber\\
&\leq \max_{B\in \mathcal{B}_1\cap \mathcal{D}_n^\perp} \norm{A'}_{(\infty)}\norm{B}_{(1)} \label{ineqrem1}\\
&\leq \frac{n}{2} \norm{A'}_{(\infty)}=\frac{n}{2}\max_{i\neq j} |a_{ij}|. \label{ineqrem2}
\end{align}
(The second inequality is by Theorem \ref{real}, which gives $\norm{B}_{(1)}\leq \frac{n}{2}$ for every $B\in\mathcal{B}_1\cap \mathcal{D}_n^\perp$.)

Now, if  $A=[a_{ij}]$, be an equality case in Theorem \ref{realdual}; that is,
\begin{equation}\label{eqrem}
\min_{D\in \mathcal{D}_n} \norm{A-D}_\infty =\frac{n}{2}\max_{i\neq j}|a_{ij}|,
\end{equation}
we must have equality in \eqref{ineqrem1} and \eqref{ineqrem2}. So there must be some $B\in \mathcal{D}_n^\perp$, such that 
\[\norm{B}_{(1)} = \frac{n}{2}\norm{B}_{1},\]
and 
\[\langle A,B\rangle =\norm{A'}_{(\infty)}\norm{B}_{(1)},\]
which means that if $M=\max_{i,j} |A'_{ij}|$, for every $i,j$ with $B_{ij}\neq 0$, $A'_{ij}=\pm{M}\mathrm{sgn}(B_{ij})$.

Therefore, all solutions to \eqref{eqrem} can be constructed as follows. Let $B=[b_{ij}]\in\mathcal{S}_n$ be a nonzero matrix with diagonal zero and $\norm{B}_{(1)} = \frac{n}{2}\norm{B}_{1}$. For an arbitrary $M\geq 0$, let $A=[a_{ij}]\in \mathcal{S}_n$  be the matrix for which,
\[\forall i\neq j, \; |a_{ij}|\leq M, \quad 
\mathrm{and \; if}\;  b_{ij}\neq 0 \; \mathrm{then}\; a_{ij}=M \,\mathrm{sgn}(b_{ij}).\]

The above construction can be applied to every introduced equality case of Theorem \ref{real} in Remark \ref{remreal}. For example if $n$ be even, we can define
\[ B=\mathbbm{1}\mathbbm{1}^*-\bf{v}\bf{v}^*, \quad\bf{v}_i=\left\{\begin{array}{ll}
1 & i\leq \frac{n}{2} \\ -1 & i>\frac{n}{2}
\end{array}\right. .\]
Then for $B$, we have $\norm{B}_{(1)} = \frac{n}{2}\norm{B}_{1}$, and
\[B=\left[\begin{matrix}
0 & 2I_{\frac{n}{2}} \\ 2I_\frac{n}{2} & 0
\end{matrix}\right].\]
So for every matrix $A$ of the form
\[A=\left[\begin{matrix}
* & I_{\frac{n}{2}} \\ I_\frac{n}{2} & *
\end{matrix}\right],\]
the equality \eqref{eqrem} holds; that is, $\min_{D\in\mathcal{D}_n} \norm{A-D}_\infty =\frac{n}{2}.$

\end{remark}

\begin{proof}[Proof of Theorem \ref{complexdual}]
The proof is completely similar to that of Theorem \ref{realdual}, and we omit the details. One shows first that Theorem \ref{complex}, implies the inclusion
\begin{equation}\label{dualc1}
\mathcal{B}_1 \cap \mathcal{D}_n^\perp \subset \cot(\frac{\pi}{2n})\mathcal{B}_{(1)},
\end{equation}
in the space $\mathcal{H}_n$. Applying the duality transformation, gives an equivalent inclusion,
\begin{equation}\label{dualc2}
\mathcal{B}_{(\infty)} \;\subset\; \cot(\frac{\pi}{2n}) \mathcal{B}_\infty+\mathcal{D}_n,
\end{equation} 
which can be rewritten as the following inequality for every $A=[a_{ij}]\in\mathcal{H}_n$,
\begin{equation}\label{dualc3}
\min_{D\in\mathcal{D}_n} \norm{A-D}_\infty \leq \cot(\frac{\pi}{2n})\max_{i\neq j} |a_{ij}|. 
\end{equation}
\end{proof}

\begin{remark}
To find the equality cases of Theorem \ref{complexdual}, a completely similar analysis to Remark \ref{remrealdual} can be made, and we only provide the final answer here.

All of solutions of the equation,
\begin{equation}\label{remf}
\min_{D\in\mathcal{D}_n} \norm{A-D}_\infty = \cot(\frac{\pi}{2n})\max_{i\neq j} |a_{ij}|,
\end{equation}
for $A\in\mathcal{H}_n$ can be constructed as follows. Let $B=[b_{ij}]\in\mathcal{H}_n$ be a nonzero matrix with diagonal zero, and $\norm{B}_{1} = \tan(\frac{\pi}{2n})\norm{B}_{(1)}$. For an arbitrary $M\geq 0$, let $A=[a_{ij}]\in \mathcal{H}_n$  be the matrix for which,
\[\forall i\neq j, \; |a_{ij}|\leq M, \;
\mathrm{and \; if}\;  b_{ij}\neq 0 \; \mathrm{then}\;  a_{ij}=M \,\frac{b_{ij}}{|b_{ij}|}.\]

For example, if we consider $B=\mathbbm{1}\mathbbm{1}^*-\bfg{\alpha}\bfg{\alpha}^*$ with the same notation in the proof of Theorem \ref{complexdual}, then we can obtain the following solution $A=[a_{ij}]$ for \eqref{remf},
\[\forall i,\; a_{ii}=0, \quad \forall i\neq j, \; a_{ij}=\frac{1-\zeta^{2(i-j)}}{|1-\zeta^{2(i-j)}|}=-\i\zeta^{i-j}\mathrm{sgn}(i-j),\]
when $\zeta=\exp(\frac{\i\pi}{n})$.

Now if we define $E\in\mathcal{H}_n$ by,
\[E_{ij}=\left\{\begin{array}{ll}
0 & i=j \\
\mathrm{sgn}(i-j)\i & i\neq j
\end{array}\right. ,\]
we have $A=-U E U^*$ for the diagonal matrix $U$ with diagonal entries 
$(\zeta,\zeta^2,\dots,\zeta^n)$. Also $U$ is a  unitary matrix and so,
\begin{align*} 
\cot(\frac{\pi}{2n})&=\min_{D\in\mathcal{D}_n} \norm{A-D}_\infty  =\min_{D\in\mathcal{D}_n} \norm{-UEU^*-D}_\infty \\
&=\min_{D\in\mathcal{D}_n} \norm{E+U^*DU}_\infty  
=\min_{D\in\mathcal{D}_n} \norm{E-D}_\infty.  
\end{align*}
Therefore, $E$ is a simple solution to \eqref{remf}. By further computation, $\norm{E}_\infty=\cot(\frac{\pi}{2n})$; thus, the nearest diagonal matrix to $E$ in the spectral norm is the zero matrix.
\end{remark}

\section{Proof of Lemma \ref{clem2}}\label{clem2sec}

\begin{proof}[Proof of Lemma \ref{clem2}.]
Define $\mathcal{X}$ to be the set of all of pairs $(\bf{d},\bfg{\omega})$ with 
\[\bf{d}=(d_1,\dots,d_n)\in\R^n_{\geq 0}, \;\sum_id_i=1,\quad 
\bfg{\omega}=(\omega_1,\dots,\omega_n)\in\S^n, \]
and define the function $F:\mathcal{X}\to\mathbb{R}$ by
\[F(\bf{d},\bfg{\omega})=\sum_{i,j} d_id_j |\omega_i-\omega_j|.\]
Also define
\begin{equation}\label{1alpha}
\mathbbm{1}=(1,\dots,1),\quad
\bfg{\alpha}=(\alpha_1,\dots,\alpha_n), \;\alpha_k=\exp(\i\frac{2(k-1)\pi}{n}), 1\leq k\leq n.
\end{equation}

Now, $\mathcal{X}$ is a compact set and  $F$ has a maximum value $M_n$ on $X$. The statement of lemma is equivalent to showing that $M_n=\gamma_n$ and 
$F(\frac{1}{n}\mathbbm{1},\bfg{\alpha})=\gamma_n$.

The proof is performed by induction on $n$. The case $n=1$ is trivial; therefore, we can assume $n>1$ and $M_{n-1}=\gamma_{n-1}$.

We have,
\begin{align*}
F(\frac{1}{n}\mathbbm{1},\bfg{\alpha}) &= 
\frac{1}{n^2}\sum_{ 1\leq i,j\leq n} |1-\bar{\alpha_i}{\alpha_j}| \\
&=  \frac{1}{n^2}\sum_{1\leq i,j\leq n} |1-\exp(\i\frac{2(j-i)\pi}{n})| \\
&= \frac{1}{n^2}\times n\sum_{k=0}^{n-1} |1-\exp(\i\frac{2k\pi}{n})| \\
&= \frac{2}{n} \sum_{k=0}^{n-1} \sin(\frac{k\pi }{n}) = \gamma_n.
\end{align*}

(By the identity $|1-\exp(\i t)|=\sqrt{2-2\cos(t)}=\sqrt{4\sin^2(\frac{t}{2})}=
2\sin(\frac{t}{2})$ for $t\in [0,2\pi]$.) Hence, $\gamma_n=F(\frac{1}{n}\mathbbm{1},\bfg{\alpha})\leq M_n$ and because $M_{n-1}=\gamma_{n-1}<\gamma_n$, $M_{n-1}<M_n$.

Now let 
\[ (\tilde{\bf{d}}=(\tilde{d}_1,\dots,\tilde{d}_n),
\tilde{\bfg{\omega}}=(\tilde{\omega}_1,\dots,\tilde{\omega}_n))\in \mathcal{X},\] be a maximum point for $F$. By a rotation and rearrangement of $\tilde{\omega}_i$'s and applying the same rearrangement to $\tilde{d}_i$'s, we can assume that $\tilde{\omega}_1=1$ and for every $1\leq k \leq n$, 
\[ \tilde{\omega}_k = \exp(2\i\theta_k),\quad 0=\theta_1\leq \theta_2\leq \dots\leq\theta_n<\pi.\]
By a sequence of steps, we prove that $(\tilde{\bf{d}},\tilde{\bfg{\omega}})=(\frac{1}{n}\mathbbm{1},\bfg{\alpha})$, and so $M_n=F(\tilde{\bf{d}},\tilde{\bfg{\omega}})=\gamma_n$, and we are done.

\vspace{10pt}\textbf{Step 1}. $\tilde{\omega}_i$'s are pairwise distinct.

 Suppose, on the contrary, $\tilde{\omega}_k=\tilde{\omega}_{k+1}$ for some $k$. Define the vectors $\bf{d}'=(d'_1,\dots,d'_{n-1})$ and $\bfg{\omega}'=(\omega'_1,\dots,\omega'_{n-1})$ by,
\[d'_i=\left\{ \begin{array}{ll} \tilde{d}_i & i< k \\ \tilde{d}_k+\tilde{d}_{k+1} & i=k \\ \tilde{d}_{i+1} & i>k \end{array}\right.,\quad
\omega'_i=\left\{ \begin{array}{ll} \tilde{\omega}_i & i\leq k \\  \tilde{\omega}_{i+1} &i>k \end{array}\right. .
\]
Then 
\begin{align*}
F(\tilde{\bf{d}},\tilde{\bfg{\omega}}) 
&= \sum_{i,j\neq k,k+1} \tilde{d}_i\tilde{d}_j|\tilde{\omega}_i-\tilde{\omega}_j| +
2\sum_{j} \tilde{d}_k\tilde{d}_j  |\tilde{\omega}_k-\tilde{\omega}_j| \\
&\quad +
2\sum_{j} \tilde{d}_{k+1}\tilde{d}_j  |\tilde{\omega}_{k+1}-\tilde{\omega}_j| \\
&= \sum_{i,j\neq k,k+1} \tilde{d}_i\tilde{d}_j|\tilde{\omega}_i-\tilde{\omega}_j| +
2\sum_{j} (\tilde{d}_k+\tilde{d}_{k+1})\tilde{d}_j  |\tilde{\omega}_k-\tilde{\omega}_j|\\
&= F(\bf{d}',\bfg{\omega}') \,\leq M_{n-1} \,<M_n,
\end{align*}
which is a contradiction.

\vspace{10pt}\textbf{Step 2}. All $\tilde{d}_i$'s are positive.

 If say $d_k$=0, it can be  easily seen that $F(\tilde{\bf{d}},\tilde{\bfg{\omega}}) $ is equal to the value of $F$ on the pair $(\bf{d}',\bfg{\omega}')$ obtained by deleting $\tilde{d}_k$ and $\tilde{\omega}_k$ from the coordinates of $\tilde{\bf{d}}$ and $\tilde{\bfg{\omega}}$. Again, $F(\bf{d}',\bfg{\omega}')\leq M_{n-1}<M_n$, which is not possible.

\vspace{10pt}\textbf{Step 3}. $\tilde{d}_1=\tilde{d}_2=\dots=\tilde{d}_n$.

For   vectors $\bfg{\omega}=(\omega_1,\dots,\omega_n)$, define the symmetric matrix $A=A(\bfg{\omega})=[|\omega_i-\omega_j|]_{i,j}$. Then we have 
\[F(\bf{d},\bfg{\omega})=\bf{d}^TA\bf{d}.\]

Because $\tilde{d}_i>0$ for every $1\leq i\leq n$, $\tilde{\bf{d}}$ is an interior point of the set $\{\bf{d}\in\R^n_{\geq 0} : \mathbbm{1}^T\bf{d}=0\}$ and maximizes the function $F(\cdot,\tilde{\bfg{\omega}})$ in this set. So by Lagrange multiplier theorem, there is a real scalar $\lambda$ such that
\[\nabla_\bf{d} F(\tilde{\bf{d}},\tilde{\bfg{\omega}})=\lambda \mathbbm{1}.\]
On the other hand,
\[\nabla_\bf{d} F(\bf{d},\bfg{\omega})=\nabla_\bf{d}(\bf{d}^TA\bf{d}) = 2A\bf{d}.\]
Hence,
\begin{equation}\label{first}
A(\tilde{\bfg{\omega}})\,\tilde{\bf{d}}=\frac{\lambda}{2} \mathbbm{1}.
\end{equation}

Now for every $i$, $\tilde{\omega}_i=\exp(2\i\theta_i)$ and $\tilde{\omega}_i$'s are pairwise distinct, so 
\[0=\theta_1<\dots<\theta_n<\pi.\] 
So in a sufficiently small neighborhood $\mathcal{B}$ of $\bfg{\theta}:=(\theta_1,\dots,\theta_n)$ in $\R^n$, every $\bfg{\sigma}=(\sigma_1,\dots,\sigma_n)\in B$ satisfies,
\[\sigma_1<\sigma_2<\dots<\sigma_n<\pi, \quad\sigma_n-\sigma_1<\pi.\]

Consider $\bfg{\omega}=(\omega_1,\dots,\omega_n)\in (\S^1)^n$ as a function of $\bfg{\sigma}=(\sigma_1,\dots,\sigma_n)\in \mathcal{B}$ by the formula,
\[\bfg{\omega}(\bfg{\sigma})=(\omega_1,\dots,\omega_n), \quad\omega_i = \exp(2\i\sigma_i), \quad 1\leq i\leq n.\]

Then for every $\bfg{\sigma}\in \mathcal{B}$, and  $\bfg{\omega} = \bfg{\omega}(\bfg{\sigma})$, there is a simpler formula for $A(\bfg{\omega})_{ij}$ in terms of $\sigma_k$'s,
\begin{align}
A(\bfg{\omega})_{ij} = |\omega_i-\omega_j|
&=|\exp(2\i\sigma_i )-\exp(2\i\sigma_j )| \nonumber\\
&= \sqrt{2-2\cos(2(\sigma_i-\sigma_j))}\nonumber\\
&=2\sin(|\sigma_i-\sigma_j|)
=\left\{\begin{array}{ll} 2\sin(\sigma_i-\sigma_j) &  i\geq j \\
2\sin(\sigma_j-\sigma_i) & i<j\end{array} \right. .
\end{align}

Now, $\tilde{\bfg{\omega}}=\bfg{\omega}(\bfg{\theta})$; hence, $\bfg{\theta}$ is a maximum for the function $F(\tilde{\bf{d}},\bfg{\omega}(\cdot))$ on the set $\mathcal{B}$. Thus derivation with respect to any of variables $\sigma_i$ for $1\leq i\leq n$, gives,

\begin{align*}
0 = \frac{\partial}{\partial \sigma_i}F(\tilde{\bf{d}},\tilde{\bfg{\omega}})
&= 2 \sum_{j=1}^n \tilde{d}_i\tilde{d}_j \frac{\partial}{\partial \sigma_i} A_{ij}(\tilde{\bfg{\omega}}) \nonumber \\
&= 4 \tilde{d}_i\left( \sum_{j<i} \tilde{d}_j \cos(\theta_i-\theta_j)  
-\sum_{j>i} \tilde{d}_j \cos(\theta_i-\theta_j)  \right),
\end{align*}

and because $\tilde{d}_i>0$ for all $i$, 
\begin{equation}\label{second}
\sum_{j<i} \tilde{d}_j \cos(\theta_i-\theta_j)  
-\sum_{j>i} \tilde{d}_j \cos(\theta_i-\theta_j) = 0, \quad 1\leq i \leq n.
\end{equation}

Now define the anti-symmetric matrix $E$ as follows,
\[E=[e_{ij}],\quad e_{ij}=\left\{\begin{array}{ll}
0 & i=j \\ -1 & i<j \\ 1 & i>j
\end{array} \right. .\]

So we can rewrite the equations \eqref{first} and \eqref{second} as bellow,
\begin{align}
\sum_{j=1}^n e_{ij}\tilde{d}_j \sin(\theta_i-\theta_j)=\frac{\lambda}{4}, \quad 1\leq i\leq n, \label{first2} \\
\sum_{j=i}^n e_{ij}\tilde{d}_j \cos(\theta_i-\theta_j)=0, \quad 1\leq i\leq n.\label{second2}  
\end{align}

Let's define
\[\phi_i = \exp(\i\theta_i), \quad 1\leq i\leq n.\]
Then we have
\[\phi_i\bar{\phi_j}=\exp(\i(\theta_i-\theta_j))
=\cos(\theta_i-\theta_j)+\i\sin(\theta_i-\theta_j).\]
Hence the equations \eqref{first2} and  \eqref{second2} are equivalent to one equation,
\begin{equation}\label{eq1}
\sum_{j=1}^n (\phi_ie_{ij}\bar{\phi_j})\,\tilde{d}_j=\frac{\i\lambda}{4}, \quad 1\leq i\leq n.
\end{equation}

After a conjugation and multiplication of both sides of the above equation by $\phi_i$ for every $i$, it becomes equivalent to the vector equation below,
\begin{equation}\label{eq}
E\,[\phi_1\tilde{d}_1\,\dots \,\phi_n\tilde{d}_n]^T= -\frac{\i\lambda}{4}\, [\phi_1 \,\dots\, \phi_n]^T.
\end{equation}

We claim that \eqref{eq} implies $\tilde{d}_1=\tilde{d}_2=\dots=\tilde{d}_n$. Let $1\leq k<n$. By \eqref{eq}, for every $i$,
\[\sum_{j=1}^n e_{ij}\tilde{d}_j\phi_j=-\frac{\i\lambda }{4} \phi_i.\]
Subtraction of the above equation for $i=k$ from the the equation for $i=k+1$, gives
\[\tilde{d}_k\phi_k+\tilde{d}_{k+1}\phi_{k+1}= -\frac{\i\lambda }{4} (\phi_{k+1}-\phi_{k}).\]
The right-hand side is perpendicular to $\phi_{k+1}-\phi_k$. But for the Euclidean inner product of the left  hand side with  $\phi_{k+1}-\phi_k$ we have,
\begin{align*}
\Re \left(
(\tilde{d}_k\phi_k+\tilde{d}_{k+1}\phi_{k+1})
\bar{(\phi_{k+1}-\phi_k)}\right) &=
\tilde{d}_{k+1}-\tilde{d}_k-
\tilde{d}_{k+1}\cos(\theta_{k+1}-\theta_k)\\
& \quad +\tilde{d}_k\cos(\theta_{k}-\theta_{k+1})\\
&= (\tilde{d}_{k+1}-\tilde{d}_k)(1-\cos(\theta_{k+1}-\theta_k)).
\end{align*}

However, $0<\theta_{k+1}-\theta_k<\pi$; therefore, $\cos(\theta_{k+1}-\theta_k)\neq 1$. Thus we have $\tilde{d}_{k+1}-\tilde{d}_k=0$, for all $k<n$ and by $\sum_i \tilde{d}_i=1$, we conclude 
\[\bf{\tilde{d}}=\frac{1}{n}\mathbbm{1}.\]
(End of Step 3.)

Now for $\mu = -\frac{n\lambda}{4}$, the equation \eqref{eq} can be simplified to
\begin{equation}\label{eq3}
E\bfg{\phi} = (\i\mu)\,\bfg{\phi},\quad  \bfg{\phi}=(\phi_1,\dots,\phi_n),
\end{equation}
which indicates that $\bfg{\phi}$ is an eigenvector of $E$.

\vspace{10pt}\textbf{Step 4}. $\tilde{\bfg{\omega}}=\bfg{\alpha}$.

Let $\zeta$ be a root of the equation $z^n=-1$. Then for every $1\leq i \leq n$, we have,
\begin{align*}
\sum_{j=1}^n e_{ij}\zeta^{j-1} &= 
\sum_{1\leq j<i} \zeta^{j-1}-\sum_{i<j\leq n} \zeta^{j-1} \\
&= -\left(\sum_{1\leq j<i} \zeta^{n+j-1}+\sum_{i<j\leq n} \zeta^{j-1} \right) \\
&= -\zeta^{i-1}\left(\zeta+\cdots+\zeta^{n-1}\right) \\
&= \zeta^{i-1} \left((-\zeta) \cdot\frac{1-\zeta^{n-1}}{1-\zeta}\right) \\
&=\zeta^{i-1} \left((-\zeta) \cdot\frac{1+\frac{1}{\zeta}}{1-\zeta}\right)
=\zeta^{i-1}\left(\frac{\zeta+1}{\zeta-1}\right).
\end{align*}

Hence 
\[E\,[1\; \zeta\; \dots \;\zeta^{n-1}]^T = \left(\frac{\zeta+1}{\zeta-1}\right)[1 \;\zeta\; \dots \;\zeta^{n-1}]^T,\]
and $(1, \zeta, \dots, \zeta^{n-1})$ is an eigenvector of $E$ for the eigenvalue $\frac{\zeta+1}{\zeta-1}$.  Now the equation $z^n=-1$ has $n$ distinct roots,
\[\zeta_i=\exp(\i\frac{(2i-1)\pi}{n}), \quad\mathrm{for}\; 1\leq i\leq n,\]

and because of the non-singularity of the Vandermonde matrix $[\zeta_i^{j-1}]_{i,j}$, the vectors $\bf{z}_i :=(1, \zeta_i, \dots, \zeta_i^{n-1})$ constitute a basis of the eigenvectors of $E$. On the other hand, the mapping $z\to \frac{z+1}{z-1}$ is one-to-one; thus, $E$ has $n$ distinct eigenvalues $\frac{\zeta_i+1}{\zeta_i-1}$.

Now $\bfg{\phi}=(\phi_1,\dots,\phi_n)$ is also an eigenvector of $E$, 
so for some $i$, $\bfg{\phi}$ is an scalar multiple of $\bf{z}_i$, 
and because $\phi_1= \exp(\i\theta_1)=1$, this scalar must be equal to $1$. Thus $\bfg{\phi}=\bf{z}_i$.

Now note that the arguments of the coordinates of $\bfg{\phi}$ are $0=\theta_1<\dots<\theta_n<\pi$. Also
 the argument of the $j$-th coordinate of $\bf{z}_i$ is 
$ a_j:=(j-1)\left(\frac{(2i-1)\pi}{n}\right)$ for $1\leq j\leq n$. For $j=2$,
$0<a_2= \frac{(2i-1)\pi}{n}<2\pi$, so $0<a_2=\theta_2<\pi$ and for every $1\leq j<n$,
\[0<a_{j+1}-a_j=\frac{(2i-1)\pi}{n} <2\pi,\quad a_{j+1}=\theta_{j+1}+2m\pi\; \mathrm{for}\;\mathrm{some}\; m\in \mathbb{Z}.\]
So by induction it can be easily seen that $a_j=\theta_j$ (for all $1\leq j\leq n$), which implies
\begin{align*}
\forall 1\leq j\leq n,\; (j-1)\left(\frac{(2i-1)\pi}{n}\right)<\pi &\;\Longrightarrow (n-1)\left(\frac{(2i-1)\pi}{n}\right)<\pi \\
&\;\Longrightarrow 2i-1<\frac{n}{n-1}\leq 2 \\
&\;\Longrightarrow i=1.
\end{align*}
So, $\bfg{\phi}=\bf{z}_1=(1, \zeta_1, \dots, \zeta_1^{n-1})$.

Now for every $1\leq i\leq n$ we have,
\[\tilde{\omega}_i=\exp(2\i\theta_i)=\phi_i^2=\zeta_1^{2(i-1)}=
\exp(\frac{2(i-1)\pi}{n})=\alpha_i.\]

Thus, $\tilde{\bfg{\omega}}=\bfg{\alpha}$ and the proof is complete.
\end{proof}

\textbf{Acknowledgements}. This work was supported in part by a grant (no. 1400050046)
from the School of Mathematics, Institute for Research in Fundamental Sciences (IPM).

\bibliographystyle{alpha}
\bibliography{Ref}

\vspace{10pt}
\noindent Mostafa  Einollahzadeh, Department of Mathematical Sciences, Isfahan University of Technology, Isfahan, Iran

\vspace{3pt}
\noindent School of Mathematics, Institute for Research in Fundamental Sciences (IPM), P.O. Box 19395-5746, Tehran, Iran

\vspace{3pt}
\noindent e-mail: m\_einollahzadeh@iut.ac.ir

%\vspace{5pt}
%\noindent
%Mohammad Mahdi Karkhaneei, Sharif University
%of Technology, Tehran, Iran
%
%\noindent e-mail: m.karkhaneei@student.sharif.edu

\end{document}